\documentclass{amsproc}
\usepackage{amsmath,amssymb,amsthm, amscd}
\usepackage{mathrsfs}
\usepackage{pb-diagram, color, pb-xy}
\usepackage[all]{xy}
\usepackage{graphicx}
\usepackage{tikz}
\usepackage{tabmac}

\usepackage{hyperref}
\usepackage{verbatim}
\usepackage{xcolor}

\theoremstyle{plain}
\newtheorem{lemma}{Lemma}[section]
\newtheorem{proposition}[lemma]{Proposition}

\newtheorem{corollary}[lemma]{Corollary}
\newtheorem{theorem}[lemma]{Theorem}

\theoremstyle{remark}
\newtheorem{openproblem}[lemma]{Open Problem}

\newtheorem{example}[lemma]{Example}

\numberwithin{equation}{section}

\newcommand{\cont}{\operatorname{cont}}
\newcommand{\Red}{\operatorname{Red}}
\newcommand{\wt}{\operatorname{wt}}
\newcommand{\et}{\tilde{e}}
\newcommand{\ft}{\tilde{f}}
\newcommand{\vp}{\varphi}
\newcommand{\ve}{\varepsilon}

\begin{document}

\title[Stanley through a crystal lens and from a random angle]{Richard Stanley through a crystal lens and from a random angle}

\author[A.~Schilling]{Anne Schilling}
\address[Anne Schilling]{Department of Mathematics, University of California, One Shields
Avenue, Davis, CA 95616-8633, U.S.A.}
\email{anne@math.ucdavis.edu}
\urladdr{http://www.math.ucdavis.edu/\~{}anne}
\thanks{Partially supported NSF grants DMS--1001256 and OCI--1147247.}

\dedicatory{Dedicated to Richard Stanley on the occasion of his seventieth birthday}



\begin{abstract}
We review Stanley's seminal work on the number of reduced words of the longest element of the symmetric
group and his Stanley symmetric functions. We shed new light on this by giving a crystal theoretic interpretation
in terms of decreasing factorizations of permutations. Whereas crystal operators on tableaux are coplactic
operators, the crystal operators on decreasing factorization intertwine with the Edelman--Greene
insertion. We also view this from a random perspective and study a Markov chain on reduced words of the longest
element in a finite Coxeter group, in particular the symmetric group, and mention a generalization to a poset setting.
\end{abstract}

\maketitle

\section{Introduction}

In his seminal paper~\cite{Stanley.1984}, Richard Stanley proved his earlier conjecture that the number of reduced words
of the longest element  $w_0$ in the symmetric group $S_n$ is equal to the number of standard Young tableaux of
staircase shape with $\binom{n}{2}$ boxes. The proof uses symmetric functions, now known as Stanley symmetric functions,
but no representation theory. According to Robert Proctor's MathSciNet review ``The proof [...] is regarded by the author as being 
somewhat mysterious.'' A different proof was given by Edelman and Greene~\cite{EG:1987} by providing an analogue
of the Robinson--Schensted--Knuth (RSK) insertion for reduced words.

Stanley symmetric functions are generating functions of decreasing factorizations of permutations.
In joint work with Jennifer Morse~\cite{MS:2014}, we define a crystal structure on these decreasing factorizations.
By using the fact that crystals of type $A_{\ell-1}$ can also be represented in terms of Young tableaux, this naturally
yields a bijection between reduced words of $w_0 \in S_n$ and standard Young tableaux of staircase shape with $\binom{n}{2}$ boxes
via the crystal isomorphism by looking at the $(1^{\binom{n}{2}})$ weight space. It turns out that this crystal isomorphism intertwines
with the Edelman--Greene insertion and hence provides a representation theoretic interpretation. Another representation
theoretic interpretation of the Stanley symmetric functions as characters of generalized Young (Specht) modules was given
in~\cite{Kr:1995,ReS:1995,ReS:1998}. In the spirit of~\cite{Assaf:2008, Haiman:1992}, the crystal on decreasing factorizations
is Schur--Weyl dual to the $S_n$-representations associated with the Coxeter--Knuth graph~\cite[Chapter 2, Section 2.2]{LLMSSZ:2014}.

We then turn to a Markov chain on the reduced words of the longest element $w_0$ of a finite Coxeter system
$(W,S)$~\cite{BB:2005} that was discovered in joint work with Arvind Ayyer, Ben Steinberg and Nicolas Thi\'ery~\cite{ASST:2014}.
The state space for this chain consists of all reduced decompositions $\Red(w_0)$ for the longest element $w_0\in W$.
The transitions or exchange moves can be defined as follows. If the system is in state $\mathfrak w=i_1 i_2 \cdots i_k$, then one 
randomly chooses a generator $s_i\in S$, appends $i$ to the beginning of the word $\mathfrak w$, and removes some other letter $i_j$ in 
$\mathfrak w$ to obtain a reduced word of $w_0$. The letter $i_j$ is uniquely determined by the exchange condition for Coxeter groups~\cite{BB:2005}.
So the exchange move goes from $\mathfrak w$ to $i \;i_1 \cdots \widehat{i_j} \cdots i_k$, where $ \widehat{i_j}$
means omit $i_j$. When $W = S_n$ is the symmetric group and $S$ is the set of adjacent transpositions, then by Stanley's result~\cite{Stanley.1984}
the reduced decompositions of $w_0$ are equi-numerous with tableaux of staircase shape. Hence this chain
can be viewed as a stochastic process on such tableaux.

The paper is organized as follows. In Section~\ref{section.stanley} we review the definition of Stanley symmetric functions
and the main results of Stanley's paper~\cite{Stanley.1984}. In Section~\ref{section.crystal} we introduce the crystal of~\cite{MS:2014}
on decreasing factorizations of any element $w\in S_n$. This is used in Section~\ref{section.stanley and crystal} to give
new crystal theoretic interpretations of the expansion coefficients of Stanley symmetric functions into Schur functions
as well as the bijection between reduced words for $w_0$ and staircase tableaux. We end in Section~\ref{section.markov}
with a description of the exchange Markov chain.

\subsection*{Acknowledgements}
I would like to thank Jennifer Morse for her collaboration on~\cite{MS:2014}, where we define the crystal structure on
decreasing factorizations, and Arvind Ayyer, Ben Steinberg and Nicolas Thi\'ery for their collaboration on~\cite{ASST:2014}
which contains the exchange Markov chains on the long element of a finite Coxeter system. Also many thanks to Thomas Lam
for interesting comments, in particular on the relation to Coxeter--Knuth graphs and the Little map.
This work benefitted from  computations with {\sc Sage}~\cite{sage,sage-combinat}. Figures~\ref{figure.crystals},~\ref{figure.Markov n=2},
and~\ref{figure.Markov n=3} were produced with {\sc Sage}. Throughout the text we provide some {\sc Sage} examples to show
how to compute the various objects appearing in this paper.

\section{Stanley symmetric functions}
\label{section.stanley}

Stanley symmetric functions are indexed by permutations $w\in S_n$ of the symmetric group. Recall that $S_n$ is
generated by the simple transpositions $s_i$ for $1\le i<n$, where each $s_i$ interchanges $i$ and $i+1$. The word
$i_1 i_2 \cdots i_\ell$ of letters $i_j\in \{ 1,2,\ldots, n-1\}$ is called a \textit{reduced word} for $w$ if
$w = s_{i_1} s_{i_2} \cdots s_{i_\ell}$ and there is no shorter word with this property. The length $\ell(w)$ of $w$ is equal to $\ell$
if the word is reduced. We denote by $w_0$ the longest element in $S_n$. Finally, we write $w \gtrdot v$ for $v,w \in S_n$ if
$w$ covers $v$ in left weak order, that is, $w=s_iv$ for some $i\in \{1,2,\ldots,n-1\}$ and $\ell(w)=\ell(v)+1$.

An element $v\in S_n$ is called \textit{decreasing} if there is a reduced word $i_1 i_2 \cdots i_\ell$ for $v$ such that 
$i_1 > i_2 > \cdots > i_\ell$. The identity is considered to be decreasing. Note that a decreasing element $v$ is completely 
determined by its content $\cont(v)$, which is the set of all letters appearing in its reduced word(s). Given $w\in S_n$, a 
\textit{decreasing factorization} of $w$ is a factorization $w^k \cdots w^1$ such that $w=w^k \cdots w^1$ with 
$\ell(w) = \ell(w^1) + \cdots + \ell(w^k)$ and each factor $w^i$ is decreasing. We denote the set of all decreasing 
factorizations of $w$ by $\mathcal{W}_w$. Then for any $w\in S_n$, the \textit{Stanley symmetric function} $F_w$ is defined as
\begin{equation}
\label{equation.stanley}
 F_w(x) = \sum_{w^k \cdots w^1\in \mathcal{W}_w} x_1^{\ell(w^1)} \cdots x_k^{\ell(w^k)}.
\end{equation}

One of Stanley's motivations to study these functions was to understand the reduced words for a given $w$.
Let us denote the set of all reduced words for $w$ by $\Red(w)$.
For example, since every single letter $i$ is decreasing, the coefficient of the square free term $x_1 x_2 \cdots x_{\ell(w)}$
is precisely the number of reduced words $|\Red(w)|$.

\begin{example}
We show how to compute the Stanley symmetric functions using {\sc Sage}~\cite{sage,sage-combinat}.
Let $w_0=s_1s_2s_1=s_2s_1s_2$ be the long element in $S_3$, which can also be viewed as the Weyl group
of type $A_2$. Then we compute:
\begin{verbatim}
    sage: W = WeylGroup(['A',2],prefix='s')
    sage: w0 = W.long_element()
    sage: w0.stanley_symmetric_function()
    2*m[1, 1, 1] + m[2, 1]
\end{verbatim}
Here $m_\lambda$ are the monomial symmetric functions. Note that the square free term(s) are contained in $m_{1,1,1}$ which
indeed has coefficient 2, the number of reduced words of $w_0$.
\end{example}

It turns out that the Stanley symmetric functions are indeed symmetric functions. One of the most important bases of the ring of symmetric
functions $\Lambda$ are the Schur functions $s_\lambda(x)$ indexed by partitions $\lambda$. Schur functions are special
as they are related to the irreducible characters of the symmetric group. Also, under the Hall inner product
$\langle \cdot, \cdot \rangle \colon \Lambda \times \Lambda \to \mathbb{C}$ the Schur functions form an orthonormal
basis $\langle s_\lambda, s_\mu\rangle = \delta_{\lambda,\mu}$. Let us denote by $\lambda^t$ the transpose of the partition
$\lambda$, which is obtained from $\lambda$ by interchanging rows and columns. Then $\omega \colon \Lambda \to \Lambda$ is
the involution such that $\omega(s_\lambda) = s_{\lambda^t}$. For $f\in \Lambda$, denote by $f^\perp : \Lambda \to \Lambda$ the linear
operator that is adjoint to multiplication under $\langle \cdot, \cdot \rangle$.

\begin{theorem}\cite{Stanley.1984}
The Stanley symmetric functions $F_w$ for $w\in S_n$ satisfy the following properties:
\begin{enumerate}
\item $F_w(x)\in \Lambda$, that is, it is a symmetric function in $x=(x_1, x_2,  \ldots)$.
\item Let $a_{w,\lambda} \in \mathbb{Z}$ be the coefficient of the Schur function $s_\lambda$ in $F_w$.
Then there exist partitions $\lambda(w)$ and $\mu(w)$, so that $a_{w,\lambda(w)} = a_{w,\mu(w)} = 1$ and 
\[
	F_w(x) = \sum_{\lambda(w)\le \lambda \le \mu(w)} a_{w,\lambda} s_\lambda(x).
\]
\item We have $\omega(F_w) = F_{w_0w}$.
\item We have $s_1^\perp F_w = \sum_{w \gtrdot v} F_v$.
\end{enumerate}
\end{theorem}

Edelman and Greene~\cite{EG:1987} and separately Lascoux and Sch\"utzenberger~\cite{LS:1985} showed that the 
Schur expansion coefficients $a_{w,\lambda}$ are nonnegative.
 
\begin{theorem} \cite{EG:1987,LS:1985}
\label{theorem.schur positive}
We have $a_{w,\lambda} \in \mathbb{Z}_{\ge 0}$. 
\end{theorem}
 
\begin{example}
We demonstrate Theorem~\ref{theorem.schur positive} in {\sc Sage}:
\begin{verbatim}
    sage: W = WeylGroup(['A',2],prefix='s')
    sage: w0 = W.long_element()
    sage: Sym = SymmetricFunctions(ZZ)
    sage: s = Sym.schur()
    sage: s(w0.stanley_symmetric_function())
    s[2, 1]
\end{verbatim}
\end{example} 
 
In the next section we will see that $a_{w,\lambda}$ can be interpreted as the number of highest weight elements of
weight $\lambda$ in a crystal graph.

\section{Crystal on decreasing factorizations}
\label{section.crystal}

This section highlights some recent joint results with Jennifer Morse~\cite{MS:2014}.
Let $\mathcal{W}_w^\ell$ be the set of all decreasing factorizations of $w\in S_n$ with $\ell$ factors (some of
which might be trivial). 
We define a crystal structure $B(w)$ of type $A_{\ell-1}$ on $\mathcal{W}_w^\ell$.

We begin by introducing an abstract crystal~\cite{Kash:1991}. Let $\mathfrak{g}$ be an (affine Kac--Moody) Lie algebra with
weight lattice $P$ and Dynkin diagram index set $I$. Denote the simple roots and simple coroots by $\alpha_i$
and $\alpha_i^\vee$ $(i \in I)$, respectively.
Then an \textit{abstract $U_q(\mathfrak{g})$-crystal} is a nonempty set $B$ together with maps
\begin{equation*}
\begin{split}
	\wt &\colon B \to P\\
	\et_i, \ft_i &\colon B \to B \cup \{\bf{0}\} \qquad \text{for all $i\in I$}
\end{split}
\end{equation*}
satisfying
\begin{enumerate}
\item $\ft_i(b)=b'$ is equivalent to $\et_i(b')=b$ for $b,b'\in B$, $i\in I$.
\item For $i\in I$ and $b\in B$
\begin{equation*}
\begin{split}
	&\wt(\et_i b) = \wt(b) +\alpha_i \qquad \text{if $\et_ib \in B$},\\
	&\wt(\ft_i b) = \wt(b) -\alpha_i \qquad \text{if $\ft_ib \in B$}.
\end{split}
\end{equation*}
\item For all $i\in I$ and $b\in B$, we have
$
	\vp_i(b) = \ve_i(b) + \langle \alpha_i^\vee, \wt(b)\rangle,
$
where
\begin{equation*}
\begin{split}
	\ve_i(b) &= \max\{d \ge 0 \mid \et_i^d(b) \neq \bf{0}\},\\
	\vp_i(b) &= \max\{d \ge 0 \mid \ft_i^d(b) \neq \bf{0}\}.
\end{split}
\end{equation*}
\end{enumerate}
An abstract crystal can be depicted by a graph, called the \textit{crystal graph}, with vertices $b\in B$ and an edge
$b \stackrel{i}{\to} b'$ if $\ft_i(b)=b'$.

The elements in $B(w)$ are the decreasing factorizations of $w$ into at most $\ell$ factors $\mathcal{W}_w^\ell$.
The weight function $\wt$ of $w^\ell \cdots w^1\in B(w)$ is defined to be $(\ell(w^1),\ell(w^2),\ldots,\ell(w^\ell))$.
The Kashiwara raising and lowering operators $\et_i$ and $\ft_i$ only act on the factors $w^{i+1} w^i$. The action is
defined by first bracketing certain letters and then moving an unbracketed letter from one factor to the other.
Let us begin by describing the bracketing procedure. Start with the largest letter $b$ in $\cont(w^{i+1})$
and pair it with the smallest $a>b$ in $\cont(w^i)$. If there is no such $a$ in $\cont(w^i)$, then
$b$ is unpaired.  The pairing proceeds in decreasing order on elements of $\cont(w^{i+1})$, and with each iteration 
previously paired letters of $\cont(w^i)$ are ignored. Define
$$
L_i(w^\ell \cdots w^1)=
\{
b\in\cont(w^{i+1}) \mid
b  \text{ is unpaired in the $w^{i+1}w^i$-pairing}
\}
$$
and
$$
R_i(w^\ell \cdots w^1)=
\{
b\in\cont(w^{i}) \mid
b  \text{ is unpaired in the $w^{i+1}w^i$-pairing}
\}\;.
$$

Then $\et_i(w^\ell \cdots w^1)$ is defined by replacing the factors $w^{i+1} w^i$ by $\widetilde w^{i+1} \widetilde w^i$ such that
$$
\cont(\widetilde w^{i+1})=\cont(w^{i+1})\backslash\{b\}\quad\text{and} 
\quad \cont(\widetilde w^i)=\cont(w^i)\cup\{b-t\}
$$
for $b=\min(L_i(w^\ell \cdots w^1))$ and
$t=\min\{j\geq 0\mid b-j-1\not\in\cont(w^{i+1})\}$.
If $L_i(w^\ell \cdots w^1)=\emptyset$, $\et_i(w^\ell \cdots w^1)=\bf{0}$.

Similarly, $\ft_i(w^\ell \cdots w^1)$ is defined by replacing the factors $w^{i+1} w^i$ by $\widetilde w^{i+1} \widetilde w^i$ such that
$$
\cont(\widetilde w^{i+1})=\cont(w^{i+1})\cup\{a+s\}\quad\text{and}\quad
\cont(\widetilde w^i)=\cont(w^i)\backslash\{a\}
$$ 
for $a=\max(R_i(w^\ell \cdots w^1))$ and
$s=\min\{j\geq 0\mid a+j+1\not\in\cont(w^i)\}$.
If $R_i(w^\ell \cdots w^1)=\emptyset$, $\ft_i(w^\ell \cdots w^1)=\bf{0}$.

\begin{example}
Let $(s_3 s_2)(s_3 s_1)(s_2) \in \mathcal{W}_w^3$ for $w= s_3 s_2s_3 s_1s_2 \in S_4$.
To apply $\et_2$ we need to first bracket the letters in $\cont(w^3) = 32$ with those in
$\cont(w^2) = 31$. The letter 3 in $\cont(w^3)$ is unbracketed since there is no bigger letter in
$\cont(w^2)$, but the letter 2 in $\cont(w^3)$ is bracketed with 3 in $\cont(w^2)$. Hence 
$b = \min(L_2(w^3 w^2 w^1))=3$ and $t=\min\{j\geq 0\mid b-j-1\not\in\cont(w^2)\}=1$.
Therefore, $\et_2((s_3 s_2)(s_3 s_1)(s_2)) = (s_2)(s_3 s_2 s_1)(s_2)$.
Similarly, $\ft_2((s_3 s_2)(s_3 s_1)(s_2)) = (s_3 s_2 s_1) (s_3)(s_2)$.
\end{example}

\begin{theorem} \cite{MS:2014}
\label{theorem.crystal}
$B(w)$ is a $U_q(A_{\ell-1})$-crystal coming from a $U_q(A_{\ell-1})$-module.
\end{theorem}

In~\cite{MS:2014}, Theorem~\ref{theorem.crystal} was proved by showing that the Stembridge local axioms are satisfied~\cite{St:2003}.

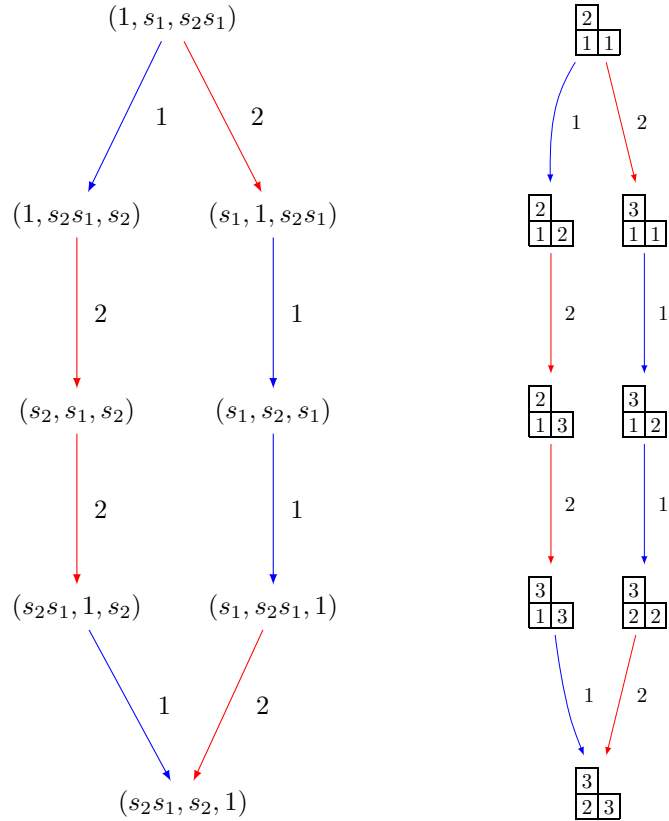
\begin{figure}
\begin{center}
\begin{tikzpicture}[>=latex,line join=bevel,]
\node (node_7) at (102.000000bp,231.000000bp) [draw,draw=none] {$\left(s_{1}, 1, s_{2}s_{1}\right)$};
  \node (node_6) at (102.000000bp,83.000000bp) [draw,draw=none] {$\left(s_{1}, s_{2}s_{1}, 1\right)$};
  \node (node_5) at (28.000000bp,83.000000bp) [draw,draw=none] {$\left(s_{2}s_{1}, 1, s_{2}\right)$};
  \node (node_4) at (68.000000bp,9.000000bp) [draw,draw=none] {$\left(s_{2}s_{1}, s_{2}, 1\right)$};
  \node (node_3) at (28.000000bp,231.000000bp) [draw,draw=none] {$\left(1, s_{2}s_{1}, s_{2}\right)$};
  \node (node_2) at (102.000000bp,157.000000bp) [draw,draw=none] {$\left(s_{1}, s_{2}, s_{1}\right)$};
  \node (node_1) at (28.000000bp,157.000000bp) [draw,draw=none] {$\left(s_{2}, s_{1}, s_{2}\right)$};
  \node (node_0) at (64.000000bp,305.000000bp) [draw,draw=none] {$\left(1, s_{1}, s_{2}s_{1}\right)$};
  \draw [blue,->] (node_0) ..controls (54.100000bp,284.650000bp) and (44.059000bp,264.010000bp)  .. (node_3);
  \definecolor{strokecol}{rgb}{0.0,0.0,0.0};
  \pgfsetstrokecolor{strokecol}
  \draw (60.000000bp,268.000000bp) node {$1$};
  \draw [blue,->] (node_7) ..controls (102.000000bp,210.870000bp) and (102.000000bp,190.800000bp)  .. (node_2);
  \draw (111.000000bp,194.000000bp) node {$1$};
  \draw [red,->] (node_1) ..controls (28.000000bp,136.870000bp) and (28.000000bp,116.800000bp)  .. (node_5);
  \draw (37.000000bp,120.000000bp) node {$2$};
  \draw [blue,->] (node_5) ..controls (39.061000bp,62.538000bp) and (50.371000bp,41.613000bp)  .. (node_4);
  \draw (61.000000bp,46.000000bp) node {$1$};
  \draw [red,->] (node_3) ..controls (28.000000bp,210.870000bp) and (28.000000bp,190.800000bp)  .. (node_1);
  \draw (37.000000bp,194.000000bp) node {$2$};
  \draw [red,->] (node_6) ..controls (92.650000bp,62.649000bp) and (83.167000bp,42.011000bp)  .. (node_4);
  \draw (98.000000bp,46.000000bp) node {$2$};
  \draw [red,->] (node_0) ..controls (74.450000bp,284.650000bp) and (85.049000bp,264.010000bp)  .. (node_7);
  \draw (96.000000bp,268.000000bp) node {$2$};
  \draw [blue,->] (node_2) ..controls (102.000000bp,136.870000bp) and (102.000000bp,116.800000bp)  .. (node_6);
  \draw (111.000000bp,120.000000bp) node {$1$};
\end{tikzpicture}
\hspace{2cm}
\scalebox{0.8}{
\begin{tikzpicture}[>=latex,line join=bevel,]
\node (node_7) at (35.000000bp,17.000000bp) [draw,draw=none] {${\def\lr#1{\multicolumn{1}{|@{\hspace{.6ex}}c@{\hspace{.6ex}}|}{\raisebox{-.3ex}{$#1$}}}\raisebox{-.6ex}{$\begin{array}[t]{*{2}c}\cline{1-1}\lr{3}\\\cline{1-2}\lr{2}&\lr{3}\\\cline{1-2}\end{array}$}}$};
  \node (node_6) at (57.000000bp,107.000000bp) [draw,draw=none] {${\def\lr#1{\multicolumn{1}{|@{\hspace{.6ex}}c@{\hspace{.6ex}}|}{\raisebox{-.3ex}{$#1$}}}\raisebox{-.6ex}{$\begin{array}[t]{*{2}c}\cline{1-1}\lr{3}\\\cline{1-2}\lr{2}&\lr{2}\\\cline{1-2}\end{array}$}}$};
  \node (node_5) at (13.000000bp,107.000000bp) [draw,draw=none] {${\def\lr#1{\multicolumn{1}{|@{\hspace{.6ex}}c@{\hspace{.6ex}}|}{\raisebox{-.3ex}{$#1$}}}\raisebox{-.6ex}{$\begin{array}[t]{*{2}c}\cline{1-1}\lr{3}\\\cline{1-2}\lr{1}&\lr{3}\\\cline{1-2}\end{array}$}}$};
  \node (node_4) at (57.000000bp,197.000000bp) [draw,draw=none] {${\def\lr#1{\multicolumn{1}{|@{\hspace{.6ex}}c@{\hspace{.6ex}}|}{\raisebox{-.3ex}{$#1$}}}\raisebox{-.6ex}{$\begin{array}[t]{*{2}c}\cline{1-1}\lr{3}\\\cline{1-2}\lr{1}&\lr{2}\\\cline{1-2}\end{array}$}}$};
  \node (node_3) at (57.000000bp,287.000000bp) [draw,draw=none] {${\def\lr#1{\multicolumn{1}{|@{\hspace{.6ex}}c@{\hspace{.6ex}}|}{\raisebox{-.3ex}{$#1$}}}\raisebox{-.6ex}{$\begin{array}[t]{*{2}c}\cline{1-1}\lr{3}\\\cline{1-2}\lr{1}&\lr{1}\\\cline{1-2}\end{array}$}}$};
  \node (node_2) at (13.000000bp,197.000000bp) [draw,draw=none] {${\def\lr#1{\multicolumn{1}{|@{\hspace{.6ex}}c@{\hspace{.6ex}}|}{\raisebox{-.3ex}{$#1$}}}\raisebox{-.6ex}{$\begin{array}[t]{*{2}c}\cline{1-1}\lr{2}\\\cline{1-2}\lr{1}&\lr{3}\\\cline{1-2}\end{array}$}}$};
  \node (node_1) at (13.000000bp,287.000000bp) [draw,draw=none] {${\def\lr#1{\multicolumn{1}{|@{\hspace{.6ex}}c@{\hspace{.6ex}}|}{\raisebox{-.3ex}{$#1$}}}\raisebox{-.6ex}{$\begin{array}[t]{*{2}c}\cline{1-1}\lr{2}\\\cline{1-2}\lr{1}&\lr{2}\\\cline{1-2}\end{array}$}}$};
  \node (node_0) at (35.000000bp,377.000000bp) [draw,draw=none] {${\def\lr#1{\multicolumn{1}{|@{\hspace{.6ex}}c@{\hspace{.6ex}}|}{\raisebox{-.3ex}{$#1$}}}\raisebox{-.6ex}{$\begin{array}[t]{*{2}c}\cline{1-1}\lr{2}\\\cline{1-2}\lr{1}&\lr{1}\\\cline{1-2}\end{array}$}}$};
  \draw [red,->] (node_0) ..controls (42.240000bp,347.380000bp) and (46.785000bp,328.790000bp)  .. (node_3);
  \definecolor{strokecol}{rgb}{0.0,0.0,0.0};
  \pgfsetstrokecolor{strokecol}
  \draw (56.000000bp,332.000000bp) node {$2$};
  \draw [blue,->] (node_5) ..controls (16.301000bp,79.448000bp) and (18.631000bp,64.720000bp)  .. (22.000000bp,52.000000bp) .. controls (22.762000bp,49.122000bp) and (23.670000bp,46.149000bp)  .. (node_7);
  \draw (31.000000bp,62.000000bp) node {$1$};
  \draw [red,->] (node_6) ..controls (49.760000bp,77.381000bp) and (45.215000bp,58.787000bp)  .. (node_7);
  \draw (56.000000bp,62.000000bp) node {$2$};
  \draw [blue,->] (node_4) ..controls (57.000000bp,167.510000bp) and (57.000000bp,149.140000bp)  .. (node_6);
  \draw (66.000000bp,152.000000bp) node {$1$};
  \draw [red,->] (node_1) ..controls (13.000000bp,257.510000bp) and (13.000000bp,239.140000bp)  .. (node_2);
  \draw (22.000000bp,242.000000bp) node {$2$};
  \draw [blue,->] (node_3) ..controls (57.000000bp,257.510000bp) and (57.000000bp,239.140000bp)  .. (node_4);
  \draw (66.000000bp,242.000000bp) node {$1$};
  \draw [blue,->] (node_0) ..controls (20.803000bp,354.630000bp) and (17.771000bp,348.250000bp)  .. (16.000000bp,342.000000bp) .. controls (13.453000bp,333.010000bp) and (12.414000bp,322.860000bp)  .. (node_1);
  \draw (25.000000bp,332.000000bp) node {$1$};
  \draw [red,->] (node_2) ..controls (13.000000bp,167.510000bp) and (13.000000bp,149.140000bp)  .. (node_5);
  \draw (22.000000bp,152.000000bp) node {$2$};
\end{tikzpicture}
}
\end{center}
\caption{
Crystal of type $A_2$ for $w_0=s_1 s_2 s_1\in S_3$ on the left and the highest weight crystal $B(2,1)$ of type $A_2$ in terms of Young tableaux on the right.
\label{figure.crystals}}
\end{figure}

\begin{example}
An example of the crystal $B(w_0)$ of type $A_2$ for $w_0\in S_3$ is provided in Figure~\ref{figure.crystals}. It can be
produced in {\sc Sage} as follows (note that type $A_2$ means 3 factors):
\begin{verbatim}
    sage: W = WeylGroup(['A',2],prefix='s')
    sage: w0 = W.long_element()
    sage: B = crystals.AffineFactorization(w0,3)
    sage: view(B)
\end{verbatim}
Here are some simple operations one can perform on this crystal:
\begin{verbatim}
    sage: B.list()
    [(1, s1, s2*s1), (1, s2*s1, s2), (s2, s1, s2),
     (s2*s1, 1, s2), (s2*s1, s2, 1), (s1, 1, s2*s1),
     (s1, s2, s1), (s1, s2*s1, 1)]
    sage: b = B.module_generators[0]; b
    (1, s1, s2*s1)
    sage: b.f(1)
    (1, s2*s1, s2)
\end{verbatim} 
\end{example}

An element $u\in B$ is called \textit{highest weight} if $\et_iu=\bf{0}$ for all $i\in I$.
A crystal $B$ is in the category of highest weight integrable crystals if for every $b\in B$, there exists a sequence
$i_1,\ldots,i_h\in I$ such that $\et_{i_1} \cdots \et_{i_h}b$ is highest weight. 
One of the most important applications of crystal theory is that 
crystals are well-behaved with respect to taking tensor products and that
connected components in a crystal graph correspond to irreducible components.

\begin{theorem} \cite{Kash:1991,N:1993,Littelmann:1994}
\label{theorem.highestwts}
Let $B$ be a $U_q(\mathfrak{g})$-crystal in the category of integrable highest-weight crystals.
Then the connected components of $B$ correspond to the irreducible components.
In addition, the irreducible components are in bijection with the highest weight vectors.
\end{theorem}

As we will see in the next section, the irreducible components of our crystal $B(w)$ on decreasing factorization are 
related to the Stanley symmetric functions.

\section{Stanley symmetric functions and crystal on decreasing factorizations}
\label{section.stanley and crystal}

We are now ready to apply crystal theory to Stanley symmetric functions.
By definition~\eqref{equation.stanley}, the Stanley symmetric function $F_w$ is the generating function of decreasing
factorizations of $w\in S_n$. By Theorem~\ref{theorem.crystal}, one can define a crystal structure on the set of decreasing
factorizations of $w$. In turn, by Theorem~\ref{theorem.highestwts} the crystal can be decomposed into irreducible
highest weight crystals $B(\lambda)$ of highest weight $\lambda$. If $B(\lambda)$ is a highest weight crystal of
type $A_{\ell-1}$, then its weight generating function is precisely the Schur polynomial indexed by $\lambda$:
\[
	s_\lambda(x_1,\ldots,x_\ell) = \sum_{b\in B(\lambda)} x^{\wt(b)}.
\]
Denote by $\mathcal W_{w,\lambda}$ all elements in $\mathcal W_w$ of weight $\lambda$.

Choosing $\ell$ sufficiently large, the above arguments immediately yield the following result.
\begin{corollary} \cite{MS:2014}
\label{corollary.a}
For any $w\in S_n$, the coefficient $a_{w,\lambda}$ in
\begin{equation}
\label{equation.F schur expansion}
    F_w = \sum_\lambda a_{w,\lambda}\, s_\lambda
\end{equation}
enumerates the highest weight factorizations in $\mathcal W_{w,\lambda}$. That is
\[
    a_{w,\lambda} = \# \{v^\ell \cdots v^1 \in \mathcal W_{w,\lambda} \mid \et_i (v^\ell \cdots v^1) = \bf{0} \text{ for all $1\le i<\ell$}\}\;.
\]
\end{corollary}

Edelman and Greene~\cite{EG:1987} (see also~\cite[Theorem 1.2]{FG:1998}) characterized 
the coefficients $a_{w,\lambda}$ as the number of semi-standard 
tableaux of shape $\lambda'$ (the transpose of $\lambda$) whose column-reading word 
is a reduced word of $w$.  This implies the following.

\begin{corollary} \cite{MS:2014}
\label{corollary.a bij}
For any permutation $w\in S_n$ and partition $\lambda$,
there is a bijection between the highest weight factorizations,
\[
	\{v^\ell \cdots v^1 \in \mathcal W_{w,\lambda} \mid \et_i (v^\ell \cdots v^1) = \bf{0} \text{ for all $1\le i<\ell$}\}\,,
\]
and the semi-standard tableaux of shape $\lambda'$ 
whose column-reading word is a reduced word of $w$.
\end{corollary}

The bijection mentioned in Corollary~\ref{corollary.a bij} is given explicitly by a variant of the Edelman--Greene (EG) 
insertion~\cite{EG:1987} (see also~\cite{BJS:1993}). In fact, it can be extended to the full crystal (not just highest weight elements) as follows. 
Given a decreasing factorization $v^\ell \cdots v^1 \in \mathcal W_{w}$, consider $\overline{v}^1 \cdots \overline{v}^\ell$ 
by reversing all factors. In particular, each decreasing factor $v^i$ turns into an increasing factor $\overline{v}^i$. 
Now successively insert the factors $\overline{v}^i$ for $i=1,2,\ldots, \ell$ using the EG insertion. In this
insertion a letter $a$ is inserted into a row by finding the smallest letter $b>a$. If $b=a+1$ and $a$ is also contained in the row, then
$a+1$ is inserted into the next row up. Otherwise, $b$ is replaced by $a$ and inserted into the next row up.
In both cases, we consider $b$ to be bumped.
For each inserted factor $\overline{v}^i$, the cells in the new shape are recorded by letters $i$. This yields
a correspondence $\varphi_{\operatorname{EG}} \colon v^\ell \cdots v^1 \mapsto (P,Q)$, where $P$ is the EG insertion tableau 
and $Q$ is the EG recording tableau. Then, according to the next theorem, the bijection of Corollary~\ref{corollary.a bij}
is explicitly the transpose of the insertion tableau $\varphi_{\operatorname{EG}}^P(v^\ell \cdots v^1)=P$ of the highest 
weight element $v^\ell \cdots v^1$.

\begin{theorem} \cite{MS:2014}
\label{theorem.intertwine}
For any permutation $w\in S_n$, the crystal isomorphism
\[
	B(w) \cong \bigoplus_\lambda B(\lambda)^{\oplus a_{w,\lambda}}
\]
is explicitly given by $\varphi_{\operatorname{EG}}^Q(v^\ell \cdots v^1) = Q$. In particular,
\[
	\varphi_{\operatorname{EG}}^Q \circ \et_i = \et_i \circ \varphi_{\operatorname{EG}}^Q 
	\qquad \text{and} \qquad
	\varphi_{\operatorname{EG}}^Q \circ \ft_i = \ft_i \circ \varphi_{\operatorname{EG}}^Q.
\]
\end{theorem}

\begin{example}
Take $v^3v^2v^1=(1)(2)(32)$ a factorization of the permutation $s_1s_2s_3s_2\in S_4$.
Then $\overline{v}^1\overline{v}^2 \overline{v}^3 = (23)(2)(1)$ with insertions for $i=1,2,3$:
\[
	\bigl( \;\tableau[scY]{2&3}\;,\; \tableau[scY]{1&1} \;\bigr) \qquad
	\bigl(\;\tableau[scY]{3\cr 2&3}\;,\; \tableau[scY]{2\cr 1&1} \;\bigr) \qquad
	\bigl(\;\tableau[scY]{3\cr 2\cr 1&3} \;,\; \tableau[scY]{3\cr 2\cr 1&1} \;\bigr)\;=\; (P,Q)\;.
\]
The element $(1)(2)(32)$ is highest weight of weight $(2,1,1)$ and the column-reading word of
the transpose of $P$
\[
	P^t = \tableau[scY]{3 \cr 1&2&3}
\]
is $3123$ which is indeed a reduced word for $s_1s_2s_3s_2$ demonstrating the 
bijective correspondence of Corollary~\ref{corollary.a bij}.
\end{example}

Another immediate outcome of our crystal $B(w)$ is Stanley's famous result~\cite{Stanley.1984} that the
number of reduced expressions for the longest element $w_0 \in S_n$ is equal to the number of standard 
tableaux of staircase shape $\rho=(n-1,n-2,\ldots,1)$. Namely, in $B(w_0)$ there is only one highest weight element
given by the factorization $(s_1)(s_2 s_1) (s_3 s_2 s_1) \cdots (s_{n-1} s_{n-2} \cdots s_1)$. Hence
$B(w_0)$ is isomorphic to the highest weight crystal $B(\rho)$. The reduced words of $w_0$ are precisely
given by the factorizations of weight $(1,1,\ldots,1)$. In $B(\rho)$ they are the standard tableaux of shape $\rho$.
The bijection between the reduced words of $w_0$ and standard tableaux of shape $\rho$ induced by the
crystal isomorphism is precisely $\varphi_{\operatorname{EG}}^Q$ (which due to the initial reversal of the factorization
gives the transpose of the standard tableau from the straight EG insertion).
An example of this crystal isomorphism for $B(s_1 s_2 s_1)$ in $S_3$ is given in Figure~\ref{figure.crystals}.

By Theorem~\ref{theorem.intertwine}, the crystal $B(w)$ relates to the crystals on the recording tableaux under the EG correspondence.
It was proved in~\cite{EG:1987} that two reduced words EG insert to the same $P$ tableau if and only if they are
Coxeter--Knuth equivalent. Two reduced words are Coxeter--Knuth equivalent if one can be obtained from the other
by a sequence of Coxeter--Knuth relations on three consecutive letters
\begin{equation}
\label{equation.ck relations}
	(a+1)a(a+1) \sim a(a+1)a, \qquad b a c \sim b c a, \qquad cab \sim acb,
\end{equation}
where the last two relations only hold when $a<b<c$. The Coxeter--Knuth graph $\mathcal{CK}(w)$ for $w\in S_n$ is a graph
on the reduced words for $w$ where two words are connected if they differ by a relation in~\eqref{equation.ck relations}.

There is an interesting relation between the crystal $B(w)$ and its decomposition into irreducible components and
the connected components of the Coxeter--Knuth graph. 
\begin{proposition}
Let $w\in S_n$. The connected components of $\mathcal{CK}(w)$ are in one-to-one correspondence with the
connected components of $B(w)$.
\end{proposition}
\begin{proof}
Every reduced word of $w$ can be viewed as an elements of $B(w)$ by placing each letter in its own factor
(assuming that $\ell$ is bigger than $\ell(w)$). Suppose $\mathfrak{w},\mathfrak{v}\in \Red(w)$ differ by a single 
relation~\ref{equation.ck relations} with $\mathfrak{w}$ having 3 consecutive letters of the left hand side and 
$\mathfrak{v}$ the corresponding letters of the right hand side. Viewing $\mathfrak{w}$ and $\mathfrak{v}$
as elements of $B(w)$, it is not hard to check that $\ft_i \ft_{i+1} \et_i \et_{i+1}(\mathfrak{w})=\mathfrak{v}$
which proves that two elements in the same component in $\mathcal{CK}(w)$ are also in the same crystal
component.

Conversely, suppose $b,b'\in B(w)$ with $\et_i(b)=b'$, so that $b$ and $b'$ lie in the same component in
$B(w)$. We can view $b$ and $b'$ as reduced words of $w$ by disregarding the grouping into factors.
By~\cite[Lemma 3.8]{MS:2014}, $b'$ is obtained from $b$ by a sequence of braid and
commutation moves. By a close inspection of the proof of~\cite[Lemma 3.8]{MS:2014} only the Coxeter--Knuth
relations~\eqref{equation.ck relations} are used. A similar argument holds for $\ft_i$.
This implies that if $b, b'\in B(w)$ are in the same component, then the corresponding
reduced words are in the same component in $\mathcal{CK}(w)$.
\end{proof}

Given that the Edelman--Greene correspondence maps a factorization to a pair of tableaux and Theorem~\ref{theorem.intertwine}
relates the crystal on decreasing factorizations to the crystal on the recording tableaux $Q$, a natural question to ask
is whether there is a ``dual'' crystal on the $P$-tableaux. By~\cite[Theorem 1.2]{HamakerYoung:2013}, two reduced words
have the same recording tableau under the EG insertion if and only if they are connected by Little bumps~\cite{Little:2005}.

\begin{openproblem}
Describe a crystal structure on the $P$-tableaux under the Edelman--Greene correspondence.
\end{openproblem}

In~\cite{MS:2014} the crystal is defined more generally on certain affine permutations into cyclically decreasing factors. Lam~\cite{Lam:2006} 
defined analogues of the Stanley symmetric functions in terms of cyclically decreasing elements.
In~\cite{MS:2014} the crystal on these affine permutations is used to study $k$-Schur structure coefficients and
further applications to flag Gromov--Witten invariances, fusion coefficients, and positroid varieties.

\section{Exchange Markov chain}
\label{section.markov}

In this section we are going to define an exchange walk on the reduced words of the longest element $w_0$
of a finite Coxeter group $W$ as first introduced in~\cite{ASST:2014}. In particular, one can consider
the symmetric group $W=S_n$.

Let $(W,S)$ be a finite Coxeter system~\cite{BB:2005}, where $W$ is a finite Coxeter group and $S=\{s_1,s_2,\ldots, s_n\}$
are simple generators. (For example, $S=\{s_1, s_2,\ldots, s_n\}$ are indeed the simple transpositions for the symmetric 
group $S_{n+1}$). 
As for the symmetric group we have the notion of reduced expressions $s_{i_1} \cdots s_{i_\ell}$ for an element 
$w\in W$ which satisfy $w=s_{i_1} \cdots s_{i_\ell}$ and there is no shorter expression with the same property.
The corresponding word $i_1 i_2 \cdots i_\ell$ is called a reduced word for $w$ and the set of all reduced words
is denoted by $\Red(w)$. Then $\ell$ is called the length $\ell(w)$ of $w$. The left weak order on $W$ is defined as follows. 
Let $w,v\in W$. We say  that $w$ covers $v$ if there exists a generator $s \in S$ such that $w=s v$ and $\ell(w)=\ell(v)+1$. 
Alternatively, a left weak cover can be characterized by requiring that $w$ has a reduced word $i_1 i_2 \ldots i_k$ such that 
$i_2 \cdots i_k$ is a reduced word for $v$. Left weak order is then the transitive closure of the cover relations.

The reduced expressions of $w_0$ can be viewed as the set of maximal chains in the left weak order.
Let $i\in I:=\{1,2,\ldots,n\}$ and $i_1 \cdots i_k \in \Red(w_0)$ a reduced word for $w_0$.
Then by the exchange condition~\cite{BB:2005}, there is a unique index $1\le j\le k$ such that $e_i(i_1\cdots i_k):=
i\; i_1\cdots \widehat{i_j}\cdots i_k$ is a reduced decomposition of $w_0$, where $i$ is added in the front and
$i_j$ is omitted.

\begin{example}
Let $123121\in \Red(w_0)$ for $w_0\in S_4$ and $i=2$. Then $e_2(123121)=2123\widehat{1}21=212321$.
\end{example}

The \textit{transition graph} is a graph whose vertices are the reduced words of $w_0\in W$ and there is an edge
labeled $i\in I$ from word $\mathfrak w\in \Red(w_0)$ to word $\mathfrak v\in \Red(w_0)$ if $e_i(\mathfrak w)=\mathfrak v$.
Examples for $n=2$ and $n=3$ are given in Figures~\ref{figure.Markov n=2} and~\ref{figure.Markov n=3}.

\begin{figure}
\begin{center}
\begin{tikzpicture}[>=latex,line join=bevel,]
\node (node_1) at (21.012000bp,9.500000bp) [draw,draw=none] {$\left(2, 1, 2\right)$};
  \node (node_0) at (21.012000bp,83.500000bp) [draw,draw=none] {$\left(1, 2, 1\right)$};
  \draw [red,->] (node_0) ..controls (8.682100bp,69.822000bp) and (4.157300bp,63.276000bp)  .. (2.012100bp,56.500000bp) .. controls (-0.670720bp,48.026000bp) and (-0.670720bp,44.974000bp)  .. (2.012100bp,36.500000bp) .. controls (3.118300bp,33.006000bp) and (4.857000bp,29.573000bp)  .. (node_1);
  \definecolor{strokecol}{rgb}{0.0,0.0,0.0};
  \pgfsetstrokecolor{strokecol}
  \draw (11.012000bp,46.500000bp) node {$2$};
  \draw [blue,->] (node_1) ..controls (21.012000bp,29.611000bp) and (21.012000bp,49.680000bp)  .. (node_0);
  \draw (30.012000bp,46.500000bp) node {$1$};
  \draw [blue,->] (node_0) ..controls (49.544000bp,90.957000bp) and (58.012000bp,88.773000bp)  .. (58.012000bp,83.500000bp) .. controls (58.012000bp,80.286000bp) and (54.867000bp,78.220000bp)  .. (node_0);
  \draw (67.012000bp,83.500000bp) node {$1$};
  \draw [red,->] (node_1) ..controls (49.544000bp,16.957000bp) and (58.012000bp,14.773000bp)  .. (58.012000bp,9.500000bp) .. controls (58.012000bp,6.286500bp) and (54.867000bp,4.220400bp)  .. (node_1);
  \draw (67.012000bp,9.500000bp) node {$2$};
\end{tikzpicture}
\end{center}
\caption{Transition graph on reduced words for $w_0\in S_3$.
\label{figure.Markov n=2}}
\end{figure}
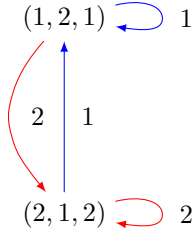

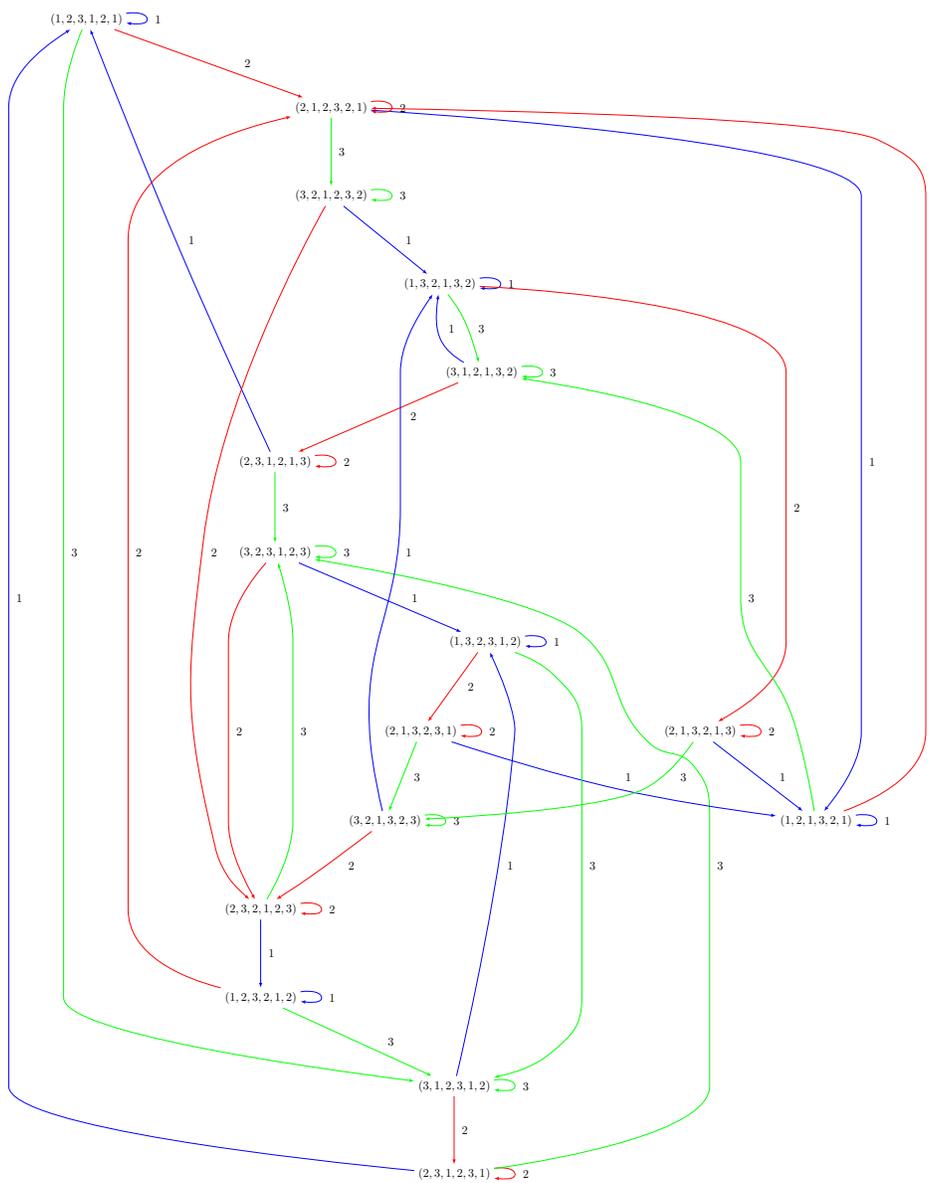
\begin{figure}
\begin{center}
\scalebox{0.45}{
\begin{tikzpicture}[>=latex,line join=bevel,]
\node (node_14) at (315.000000bp,305.500000bp) [draw,draw=none] {$\left(3, 2, 1, 3, 2, 3\right)$};
  \node (node_3) at (361.000000bp,755.500000bp) [draw,draw=none] {$\left(1, 3, 2, 1, 3, 2\right)$};
  \node (node_9) at (373.000000bp,9.500000bp) [draw,draw=none] {$\left(2, 3, 1, 2, 3, 1\right)$};
  \node (node_8) at (223.000000bp,606.500000bp) [draw,draw=none] {$\left(2, 3, 1, 2, 1, 3\right)$};
  \node (node_7) at (345.000000bp,380.500000bp) [draw,draw=none] {$\left(2, 1, 3, 2, 3, 1\right)$};
  \node (node_6) at (579.000000bp,380.500000bp) [draw,draw=none] {$\left(2, 1, 3, 2, 1, 3\right)$};
  \node (node_5) at (270.000000bp,903.500000bp) [draw,draw=none] {$\left(2, 1, 2, 3, 2, 1\right)$};
  \node (node_4) at (399.000000bp,455.500000bp) [draw,draw=none] {$\left(1, 3, 2, 3, 1, 2\right)$};
  \node (node_13) at (270.000000bp,829.500000bp) [draw,draw=none] {$\left(3, 2, 1, 2, 3, 2\right)$};
  \node (node_2) at (211.000000bp,157.500000bp) [draw,draw=none] {$\left(1, 2, 3, 2, 1, 2\right)$};
  \node (node_1) at (65.000000bp,977.500000bp) [draw,draw=none] {$\left(1, 2, 3, 1, 2, 1\right)$};
  \node (node_0) at (676.000000bp,305.500000bp) [draw,draw=none] {$\left(1, 2, 1, 3, 2, 1\right)$};
  \node (node_11) at (396.000000bp,681.500000bp) [draw,draw=none] {$\left(3, 1, 2, 1, 3, 2\right)$};
  \node (node_10) at (211.000000bp,231.500000bp) [draw,draw=none] {$\left(2, 3, 2, 1, 2, 3\right)$};
  \node (node_15) at (223.000000bp,530.500000bp) [draw,draw=none] {$\left(3, 2, 3, 1, 2, 3\right)$};
  \node (node_12) at (373.000000bp,83.500000bp) [draw,draw=none] {$\left(3, 1, 2, 3, 1, 2\right)$};
  \draw [blue,->] (node_4) ..controls (442.280000bp,462.030000bp) and (450.000000bp,459.790000bp)  .. (450.000000bp,455.500000bp) .. controls (450.000000bp,452.820000bp) and (446.980000bp,450.940000bp)  .. (node_4);
  \definecolor{strokecol}{rgb}{0.0,0.0,0.0};
  \pgfsetstrokecolor{strokecol}
  \draw (459.000000bp,455.500000bp) node {$1$};
  \draw [red,->] (node_9) ..controls (416.280000bp,16.034000bp) and (424.000000bp,13.789000bp)  .. (424.000000bp,9.500000bp) .. controls (424.000000bp,6.819300bp) and (420.980000bp,4.937100bp)  .. (node_9);
  \draw (433.000000bp,9.500000bp) node {$2$};
  \draw [green,->] (node_9) ..controls (466.620000bp,23.006000bp) and (587.000000bp,45.761000bp)  .. (587.000000bp,83.500000bp) .. controls (587.000000bp,305.500000bp) and (587.000000bp,305.500000bp)  .. (587.000000bp,305.500000bp) .. controls (587.000000bp,327.300000bp) and (587.010000bp,335.810000bp)  .. (573.000000bp,352.500000bp) .. controls (561.500000bp,366.200000bp) and (550.290000bp,358.530000bp)  .. (537.000000bp,370.500000bp) .. controls (500.170000bp,403.670000bp) and (516.460000bp,434.500000bp)  .. (477.000000bp,464.500000bp) .. controls (444.020000bp,489.580000bp) and (329.700000bp,512.330000bp)  .. (node_15);
  \draw (596.000000bp,268.500000bp) node {$3$};
  \draw [red,->] (node_11) ..controls (345.650000bp,659.670000bp) and (287.090000bp,634.280000bp)  .. (node_8);
  \draw (339.000000bp,644.500000bp) node {$2$};
  \draw [red,->] (node_8) ..controls (266.280000bp,613.400000bp) and (274.000000bp,611.030000bp)  .. (274.000000bp,606.500000bp) .. controls (274.000000bp,603.670000bp) and (270.980000bp,601.680000bp)  .. (node_8);
  \draw (283.000000bp,606.500000bp) node {$2$};
  \draw [green,->] (node_10) ..controls (224.350000bp,253.710000bp) and (238.000000bp,280.610000bp)  .. (238.000000bp,305.500000bp) .. controls (238.000000bp,455.500000bp) and (238.000000bp,455.500000bp)  .. (238.000000bp,455.500000bp) .. controls (238.000000bp,475.160000bp) and (232.980000bp,497.170000bp)  .. (node_15);
  \draw (247.000000bp,380.500000bp) node {$3$};
  \draw [blue,->] (node_5) ..controls (405.910000bp,894.760000bp) and (714.000000bp,871.160000bp)  .. (714.000000bp,829.500000bp) .. controls (714.000000bp,829.500000bp) and (714.000000bp,829.500000bp)  .. (714.000000bp,380.500000bp) .. controls (714.000000bp,358.500000bp) and (700.830000bp,336.410000bp)  .. (node_0);
  \draw (723.000000bp,606.500000bp) node {$1$};
  \draw [blue,->] (node_11) ..controls (372.980000bp,694.610000bp) and (365.820000bp,700.800000bp)  .. (362.000000bp,708.500000bp) .. controls (357.770000bp,717.030000bp) and (357.190000bp,727.560000bp)  .. (node_3);
  \draw (371.000000bp,718.500000bp) node {$1$};
  \draw [green,->] (node_8) ..controls (223.000000bp,585.990000bp) and (223.000000bp,564.630000bp)  .. (node_15);
  \draw (232.000000bp,568.500000bp) node {$3$};
  \draw [blue,->] (node_3) ..controls (404.280000bp,762.030000bp) and (412.000000bp,759.790000bp)  .. (412.000000bp,755.500000bp) .. controls (412.000000bp,752.820000bp) and (408.980000bp,750.940000bp)  .. (node_3);
  \draw (421.000000bp,755.500000bp) node {$1$};
  \draw [blue,->] (node_9) ..controls (248.610000bp,21.511000bp) and (0.000000bp,49.005000bp)  .. (0.000000bp,83.500000bp) .. controls (0.000000bp,903.500000bp) and (0.000000bp,903.500000bp)  .. (0.000000bp,903.500000bp) .. controls (0.000000bp,929.330000bp) and (22.637000bp,950.450000bp)  .. (node_1);
  \draw (9.000000bp,492.500000bp) node {$1$};
  \draw [blue,->] (node_10) ..controls (211.000000bp,211.370000bp) and (211.000000bp,191.300000bp)  .. (node_2);
  \draw (220.000000bp,194.500000bp) node {$1$};
  \draw [blue,->] (node_15) ..controls (274.220000bp,508.670000bp) and (333.800000bp,483.280000bp)  .. (node_4);
  \draw (340.000000bp,492.500000bp) node {$1$};
  \draw [green,->] (node_2) ..controls (258.190000bp,135.940000bp) and (311.730000bp,111.490000bp)  .. (node_12);
  \draw (320.000000bp,120.500000bp) node {$3$};
  \draw [blue,->] (node_6) ..controls (606.430000bp,359.290000bp) and (637.010000bp,335.640000bp)  .. (node_0);
  \draw (648.000000bp,342.500000bp) node {$1$};
  \draw [red,->] (node_3) ..controls (470.440000bp,749.710000bp) and (651.000000bp,733.980000bp)  .. (651.000000bp,681.500000bp) .. controls (651.000000bp,681.500000bp) and (651.000000bp,681.500000bp)  .. (651.000000bp,455.500000bp) .. controls (651.000000bp,428.300000bp) and (626.050000bp,407.130000bp)  .. (node_6);
  \draw (660.000000bp,568.500000bp) node {$2$};
  \draw [blue,->] (node_13) ..controls (296.070000bp,808.300000bp) and (324.170000bp,785.450000bp)  .. (node_3);
  \draw (335.000000bp,792.500000bp) node {$1$};
  \draw [green,->] (node_14) ..controls (358.280000bp,312.030000bp) and (366.000000bp,309.790000bp)  .. (366.000000bp,305.500000bp) .. controls (366.000000bp,302.820000bp) and (362.980000bp,300.940000bp)  .. (node_14);
  \draw (375.000000bp,305.500000bp) node {$3$};
  \draw [blue,->] (node_1) ..controls (108.280000bp,984.030000bp) and (116.000000bp,981.790000bp)  .. (116.000000bp,977.500000bp) .. controls (116.000000bp,974.820000bp) and (112.980000bp,972.940000bp)  .. (node_1);
  \draw (125.000000bp,977.500000bp) node {$1$};
  \draw [green,->] (node_13) ..controls (313.280000bp,836.030000bp) and (321.000000bp,833.790000bp)  .. (321.000000bp,829.500000bp) .. controls (321.000000bp,826.820000bp) and (317.980000bp,824.940000bp)  .. (node_13);
  \draw (330.000000bp,829.500000bp) node {$3$};
  \draw [green,->] (node_0) ..controls (672.040000bp,329.590000bp) and (665.860000bp,363.170000bp)  .. (657.000000bp,390.500000bp) .. controls (641.770000bp,437.460000bp) and (613.000000bp,443.130000bp)  .. (613.000000bp,492.500000bp) .. controls (613.000000bp,606.500000bp) and (613.000000bp,606.500000bp)  .. (613.000000bp,606.500000bp) .. controls (613.000000bp,642.840000bp) and (502.900000bp,665.460000bp)  .. (node_11);
  \draw (622.000000bp,492.500000bp) node {$3$};
  \draw [blue,->] (node_12) ..controls (383.380000bp,127.370000bp) and (413.420000bp,259.640000bp)  .. (423.000000bp,370.500000bp) .. controls (423.770000bp,379.360000bp) and (424.560000bp,381.750000bp)  .. (423.000000bp,390.500000bp) .. controls (420.060000bp,406.970000bp) and (413.230000bp,424.720000bp)  .. (node_4);
  \draw (420.000000bp,268.500000bp) node {$1$};
  \draw [blue,->] (node_0) ..controls (719.280000bp,312.030000bp) and (727.000000bp,309.790000bp)  .. (727.000000bp,305.500000bp) .. controls (727.000000bp,302.820000bp) and (723.980000bp,300.940000bp)  .. (node_0);
  \draw (736.000000bp,305.500000bp) node {$1$};
  \draw [blue,->] (node_7) ..controls (403.000000bp,361.480000bp) and (460.090000bp,343.770000bp)  .. (510.000000bp,332.500000bp) .. controls (551.300000bp,323.170000bp) and (599.070000bp,315.750000bp)  .. (node_0);
  \draw (519.000000bp,342.500000bp) node {$1$};
  \draw [blue,->] (node_8) ..controls (207.080000bp,640.820000bp) and (171.610000bp,717.810000bp)  .. (144.000000bp,782.500000bp) .. controls (116.440000bp,847.070000bp) and (85.733000bp,924.600000bp)  .. (node_1);
  \draw (153.000000bp,792.500000bp) node {$1$};
  \draw [red,->] (node_1) ..controls (125.410000bp,955.690000bp) and (195.110000bp,930.530000bp)  .. (node_5);
  \draw (200.000000bp,940.500000bp) node {$2$};
  \draw [green,->] (node_12) ..controls (416.280000bp,90.034000bp) and (424.000000bp,87.789000bp)  .. (424.000000bp,83.500000bp) .. controls (424.000000bp,80.819000bp) and (420.980000bp,78.937000bp)  .. (node_12);
  \draw (433.000000bp,83.500000bp) node {$3$};
  \draw [blue,->] (node_2) ..controls (254.280000bp,164.030000bp) and (262.000000bp,161.790000bp)  .. (262.000000bp,157.500000bp) .. controls (262.000000bp,154.820000bp) and (258.980000bp,152.940000bp)  .. (node_2);
  \draw (271.000000bp,157.500000bp) node {$1$};
  \draw [green,->] (node_5) ..controls (270.000000bp,883.370000bp) and (270.000000bp,863.300000bp)  .. (node_13);
  \draw (279.000000bp,866.500000bp) node {$3$};
  \draw [red,->] (node_0) ..controls (729.320000bp,324.910000bp) and (768.000000bp,346.290000bp)  .. (768.000000bp,380.500000bp) .. controls (768.000000bp,829.500000bp) and (768.000000bp,829.500000bp)  .. (768.000000bp,829.500000bp) .. controls (768.000000bp,856.930000bp) and (752.550000bp,864.270000bp)  .. (728.000000bp,876.500000bp) .. controls (691.520000bp,894.680000bp) and (420.590000bp,901.050000bp)  .. (node_5);
  \draw (777.000000bp,606.500000bp) node {$2$};
  \draw [green,->] (node_6) ..controls (565.360000bp,360.440000bp) and (550.190000bp,341.510000bp)  .. (532.000000bp,332.500000bp) .. controls (502.360000bp,317.820000bp) and (412.510000bp,310.610000bp)  .. (node_14);
  \draw (565.000000bp,342.500000bp) node {$3$};
  \draw [red,->] (node_2) ..controls (145.280000bp,173.950000bp) and (100.000000bp,193.630000bp)  .. (100.000000bp,231.500000bp) .. controls (100.000000bp,792.500000bp) and (100.000000bp,792.500000bp)  .. (100.000000bp,792.500000bp) .. controls (100.000000bp,853.010000bp) and (176.450000bp,882.280000bp)  .. (node_5);
  \draw (109.000000bp,530.500000bp) node {$2$};
  \draw [green,->] (node_4) ..controls (436.510000bp,442.640000bp) and (447.710000bp,436.620000bp)  .. (456.000000bp,428.500000bp) .. controls (473.040000bp,411.810000bp) and (480.000000bp,404.350000bp)  .. (480.000000bp,380.500000bp) .. controls (480.000000bp,380.500000bp) and (480.000000bp,380.500000bp)  .. (480.000000bp,157.500000bp) .. controls (480.000000bp,133.410000bp) and (471.620000bp,125.790000bp)  .. (453.000000bp,110.500000bp) .. controls (442.350000bp,101.750000bp) and (428.730000bp,95.790000bp)  .. (node_12);
  \draw (489.000000bp,268.500000bp) node {$3$};
  \draw [red,->] (node_15) ..controls (203.280000bp,508.260000bp) and (184.000000bp,481.920000bp)  .. (184.000000bp,455.500000bp) .. controls (184.000000bp,455.500000bp) and (184.000000bp,455.500000bp)  .. (184.000000bp,305.500000bp) .. controls (184.000000bp,285.280000bp) and (193.010000bp,263.730000bp)  .. (node_10);
  \draw (193.000000bp,380.500000bp) node {$2$};
  \draw [red,->] (node_13) ..controls (245.970000bp,786.800000bp) and (177.050000bp,657.360000bp)  .. (163.000000bp,540.500000bp) .. controls (150.050000bp,432.790000bp) and (145.230000bp,402.120000bp)  .. (170.000000bp,296.500000bp) .. controls (174.110000bp,278.980000bp) and (174.300000bp,273.660000bp)  .. (184.000000bp,258.500000bp) .. controls (186.560000bp,254.500000bp) and (189.810000bp,250.620000bp)  .. (node_10);
  \draw (172.000000bp,530.500000bp) node {$2$};
  \draw [green,->] (node_3) ..controls (371.920000bp,741.790000bp) and (376.670000bp,735.040000bp)  .. (380.000000bp,728.500000bp) .. controls (384.610000bp,719.430000bp) and (388.370000bp,708.740000bp)  .. (node_11);
  \draw (396.000000bp,718.500000bp) node {$3$};
  \draw [red,->] (node_5) ..controls (313.280000bp,910.030000bp) and (321.000000bp,907.790000bp)  .. (321.000000bp,903.500000bp) .. controls (321.000000bp,900.820000bp) and (317.980000bp,898.940000bp)  .. (node_5);
  \draw (330.000000bp,903.500000bp) node {$2$};
  \draw [red,->] (node_4) ..controls (384.100000bp,434.800000bp) and (368.070000bp,412.540000bp)  .. (node_7);
  \draw (387.000000bp,418.500000bp) node {$2$};
  \draw [green,->] (node_7) ..controls (336.860000bp,360.140000bp) and (328.310000bp,338.780000bp)  .. (node_14);
  \draw (342.000000bp,342.500000bp) node {$3$};
  \draw [green,->] (node_15) ..controls (266.280000bp,537.400000bp) and (274.000000bp,535.030000bp)  .. (274.000000bp,530.500000bp) .. controls (274.000000bp,527.670000bp) and (270.980000bp,525.680000bp)  .. (node_15);
  \draw (283.000000bp,530.500000bp) node {$3$};
  \draw [red,->] (node_6) ..controls (622.280000bp,387.400000bp) and (630.000000bp,385.030000bp)  .. (630.000000bp,380.500000bp) .. controls (630.000000bp,377.670000bp) and (626.980000bp,375.680000bp)  .. (node_6);
  \draw (639.000000bp,380.500000bp) node {$2$};
  \draw [green,->] (node_1) ..controls (55.603000bp,954.860000bp) and (46.000000bp,927.640000bp)  .. (46.000000bp,903.500000bp) .. controls (46.000000bp,903.500000bp) and (46.000000bp,903.500000bp)  .. (46.000000bp,157.500000bp) .. controls (46.000000bp,128.740000bp) and (240.970000bp,100.270000bp)  .. (node_12);
  \draw (55.000000bp,530.500000bp) node {$3$};
  \draw [red,->] (node_12) ..controls (373.000000bp,63.372000bp) and (373.000000bp,43.301000bp)  .. (node_9);
  \draw (382.000000bp,46.500000bp) node {$2$};
  \draw [red,->] (node_7) ..controls (388.280000bp,387.400000bp) and (396.000000bp,385.030000bp)  .. (396.000000bp,380.500000bp) .. controls (396.000000bp,377.670000bp) and (392.980000bp,375.680000bp)  .. (node_7);
  \draw (405.000000bp,380.500000bp) node {$2$};
  \draw [green,->] (node_11) ..controls (439.280000bp,688.030000bp) and (447.000000bp,685.790000bp)  .. (447.000000bp,681.500000bp) .. controls (447.000000bp,678.820000bp) and (443.980000bp,676.940000bp)  .. (node_11);
  \draw (456.000000bp,681.500000bp) node {$3$};
  \draw [red,->] (node_14) ..controls (291.530000bp,287.410000bp) and (270.550000bp,271.510000bp)  .. (252.000000bp,258.500000bp) .. controls (245.910000bp,254.230000bp) and (239.200000bp,249.720000bp)  .. (node_10);
  \draw (287.000000bp,268.500000bp) node {$2$};
  \draw [blue,->] (node_14) ..controls (309.940000bp,326.680000bp) and (304.860000bp,350.160000bp)  .. (303.000000bp,370.500000bp) .. controls (294.910000bp,458.830000bp) and (328.000000bp,479.800000bp)  .. (328.000000bp,568.500000bp) .. controls (328.000000bp,681.500000bp) and (328.000000bp,681.500000bp)  .. (328.000000bp,681.500000bp) .. controls (328.000000bp,702.500000bp) and (339.220000bp,724.180000bp)  .. (node_3);
  \draw (335.000000bp,530.500000bp) node {$1$};
  \draw [red,->] (node_10) ..controls (254.280000bp,238.030000bp) and (262.000000bp,235.790000bp)  .. (262.000000bp,231.500000bp) .. controls (262.000000bp,228.820000bp) and (258.980000bp,226.940000bp)  .. (node_10);
  \draw (271.000000bp,231.500000bp) node {$2$};
\end{tikzpicture}
}
\end{center}
\caption{Transition graph on reduced words for $w_0\in S_4$.
\label{figure.Markov n=3}}
\end{figure}

We can now consider a Markov chain on the transition graph by defining a probability measure $P$ on $I$
which we call the \emph{exchange walk on $(W,S)$}. The probability measure $P$ is a map
\begin{equation*}
\begin{split}
	P \colon I &\to [0,1] \subset \mathbb{R}\\
	    i &\mapsto P(i)
\end{split}
\end{equation*}
such that $\sum_{i \in I} P(i)=1$. The state space of the Markov chain is $\Red(w_0)$.  Transitions are given by changing 
from state $\mathfrak w$ to state $e_i(\mathfrak w)$ with probability $P(i)$.  

The \textit{transition matrix} $T$ of a discrete-time Markov chain is the matrix, whose entries are indexed by elements of the 
state space. In our case they are labeled by elements in $\Red(w_0)$. We take the convention that the 
$(\mathfrak w', \mathfrak w)$-entry gives the probability of going from $\mathfrak w \to \mathfrak w'$. The special case of the diagonal entry at 
$(\mathfrak w,\mathfrak w)$ gives the probability of a loop at the $\mathfrak w$. This ensures that column sums of $T$ are one and 
consequently, one is an eigenvalue with row (left-) eigenvector being the all-ones vector.
The \textit{stationary distribution} $\pi$ is the (right-) eigenvector of $T$ with eigenvalue 1, that is $T\pi=\pi$.
A Markov chain is irreducible if the transition graph, with vertices in the state space
and an edge from $\mathfrak w$ to $\mathfrak w'$ if $T(\mathfrak w,\mathfrak w')>0$, is strongly connected.
A Markov chain is \textit{ergodic} if it is irreducible and the greatest common divisor of all cycle lengths is one.
Typical questions regarding Markov chains ask for the eigenvalues of the transition matrix, the stationary
distributions, and their mixing times.

To state the main result, we need some notation.  Let $W_J=\langle s_j \mid j \in J\rangle$ be the standard parabolic
subgroup associated to $J\subseteq I$. Let
\[
	D_R(w) = \{i\in I\mid \ell(ws_i)<\ell(w)\}
\]
be the set of right descents of $w\in W$. Let $w_J$ denote the longest element of $W_J$; note that $w_J$
is an involution and $w_I=w_0$.

\begin{theorem}\label{theorem.exchange} \cite{ASST:2014}
Let $(W,S)$ be a finite Coxeter system with $S=\{s_i \mid i\in I\}$, and let $P$ be a probability measure on $I$ with support $I$.  Let $T$
be the transition matrix of the exchange chain on $(W,S)$. Then the exchange chain is ergodic and the following hold.
\begin{enumerate}
\item The eigenvalues of $T$ are
\[
	\lambda_J = \sum_{j\in J}P(j),
\]
where $J\subseteq I$.
\item The multiplicity of $\lambda_J$ as an eigenvalue is given by
\[
	\sum_{K\supseteq J}(-1)^{|K|-|J|}\cdot |\Red(w_Jw_0)|.
\]
\item The stationary distribution $\pi$ is given as follows: if $\mathfrak w =i_1\cdots i_\ell \in \Red(w_0)$, then
\[
	\pi(\mathfrak w) = \prod_{j=1}^\ell \frac{P(i_j)}{1-\lambda_{D_R(s_{i_1} \cdots s_{i_{j-1}})}}.
\]
\end{enumerate}
\end{theorem}

Given the crystal isomorphism $B(\omega_0) \cong B(\rho)$ of Section~\ref{section.stanley and crystal}
for $w_0\in S_n$, it is natural to ask whether there is a combinatorial description of the Markov chain
in terms of standard staircase tableaux.

\begin{openproblem}
Can the exchange Markov chain for $S_n$ be naturally described on the set of standard staircase tableaux
with  $\binom{n}{2}$ boxes?
\end{openproblem} 

When $W=(\mathbb Z/2\mathbb Z)^n$ with $S$ the standard unit vectors,
then $w_0$ is the all-ones vector and the reduced decompositions of $w_0$ are all linear orderings of $S$
(written as words in the alphabet $I=\{1,2,\ldots,n\}$).  Then $e_i(\mathfrak w)$ moves the letter $i$ to the front of the word
$\mathfrak w$ (or linear ordering of $S$). The corresponding Markov chain is known as the \textit{Tsetlin library}.
It can be viewed as a shelf in the library with $n$ books labeled $1,2,\ldots,n$. The books can be arranged in
any order, that is permutations of $n$ letters (or as above linear orderings of $S$). A book can be chosen at random
and is then returned to the front of the bookshelf.

In~\cite{AKS:2014} a generalization of the Tsetlin library was defined using a generalized promotion operator on posets
of another one of Stanley's very interesting papers~\cite{Stanley.2009}. Start with a finite poset $P$ with a natural labeling, that
is, if $x<y$ in $P$, then the integer label $i$ at vertex $x$ and label $j$ at $y$ should also satisfy $i<j$.
One can define simple transposition operators $\tau_i$ on the linear extensions of $\mathcal L(P)$.
Namely, for $\pi \in \mathcal{L}(P)$ the operator $\tau_i$ interchanges $\pi_i$ and $\pi_{i+1}$ if $\pi_i$
and $\pi_{i+1}$ are not comparable in $P$. Otherwise it acts as the identity.
Then the generalized promotion operator is $\partial_i(\pi)= \tau_1 \tau_2 \cdots \tau_{i-1}(\pi)$.
 Set $\widehat{\partial}_i(\pi) = \partial_{\pi^{-1}_i}(\pi)$. Assigning a probability $P(i)$ to the promotion operator
 $\widehat{\partial}_i$ defines a Markov chain. As in Theorem~\ref{theorem.exchange}, the eigenvalues and
 their multiplicities can be computed explicitly for special posets (rooted forests) using the representation theory
 of $\mathscr{R}$-trivial monoids~\cite{AKS:2014,ASST:2014}.


\end{document}